\documentclass[11pt]{amsart}

\usepackage{latexsym}
\usepackage{amssymb}
\usepackage{amsfonts}
\usepackage{graphicx}
\usepackage{epsf}
\usepackage[all]{xypic}

\usepackage{amsthm}

\usepackage{amsmath}
\usepackage{amscd}

\newtheorem{theorem}{Theorem}[section]
\newtheorem{corollary}[theorem]{Corollary}
\newtheorem{lemma}[theorem]{Lemma}
\newtheorem{proposition}[theorem]{Proposition}

\newtheorem{example}[theorem]{Example}
\newtheorem{remark}[theorem]{Remark}

\def\bea{\begin{eqnarray*}}
\def\eea{\end{eqnarray*}}

\def\cosmash{\mathrel\rhd\joinrel\mathrel<}

\def\ot{\otimes}
\def\ra{\rightarrow}

\def\bea{\begin{eqnarray*}}
\def\eea{\end{eqnarray*}}

\def\rhu{\rightharpoonup}
\def\lhu{\leftharpoonup}

\begin{document}

\title[Graded Frobenius algebras from tensor algebras of bimodules]{Graded Frobenius algebras from tensor algebras of bimodules}
\author[S. D\u{a}sc\u{a}lescu, C. N\u{a}st\u{a}sescu and L. N\u{a}st\u{a}sescu]{ Sorin D\u{a}sc\u{a}lescu$^1$, Constantin
N\u{a}st\u{a}sescu$^2$ and Laura N\u{a}st\u{a}sescu$^{2}$}
\address{$^1$ University of Bucharest, Faculty of Mathematics and Computer Science,
Str. Academiei 14, Bucharest 1, RO-010014, Romania} \address{ $^2$
Institute of Mathematics of the Romanian Academy, PO-Box
1-764\\
RO-014700, Bucharest, Romania}
\address{
 e-mail: sdascal@fmi.unibuc.ro, Constantin\_nastasescu@yahoo.com,
 lauranastasescu@gmail.com
}

\date{}
\maketitle

\begin{abstract}
We consider certain quotient algebras of tensor algebras of
bimodules $M$ over a finite-dimensional algebra $R$, and we
investigate Frobenius type properties of such algebras. Our main
interest is in the case where $M=R^*$, the linear dual of $R$. We
obtain a large class of Frobenius or symmetric algebras, which are
also equipped with a finite grading.
\end{abstract}

\section{Introduction and preliminaries}

Trivial extensions associated with bimodules over a ring have
occurred as relevant objects in several branches of algebra: in
homological algebra (Hochschild's study of the cohomology of a ring
with coefficients in a bimodule, Quillen's cohomological
investigation of the category of rings over a given ring), in
algebraic geometry (the ring of dual numbers and its usefulness to
schemes), in commutative rings (Nagata's principle of idealization,
the theory of Gorenstein modules), in representation theory (the
trivial extension of an algebra); see \cite{fgr} for details. If $M$
is a bimodule over a ring $R$, the trivial extension $R\cosmash M$
is the additive group $R\oplus M$ with a multiplication defined by
$(r_1,m_1)\cdot (r_2,m_2)=(r_1r_2, r_1m_2+m_1r_2)$, which makes it a
ring.

A striking example in representation theory, due to Tachikawa, shows
that if $R$ is an arbitrary finite-dimensional algebra, and $R^*$ is
its linear dual, endowed with the usual $R$-bimodule structure, then
the trivial extension $R\cosmash R^*$ is a symmetric algebra, see
\cite[16.60]{lam}; consequently, any finite dimensional algebra is
isomorphic to a factor algebra of a symmetric algebra. We aim to
uncover what is behind this example.

A more general construction is the semitrivial extension
$R\cosmash_\varphi M$ associated with an $R$-bimodule $M$ and a
morphism of bimodules $\varphi:M\ot_RM\ra R$ which is associative,
i.e., it satisfies $\varphi(m_1\ot_Rm_2)m_3=m_1\varphi(m_2\ot_Rm_3)$
for any $m_1,m_2,m_3$. The multiplication is defined in this case by
$(r_1,m_1)\cdot (r_2,m_2)=(r_1r_2+\varphi(m_1\ot_Rm_2),
r_1m_2+m_1r_2)$, see \cite{palmer}. It was proved in \cite{dnn2}
that if $R$ is a finite-dimensional algebra, then any semitrivial
extension $R\cosmash_{\varphi}R^*$ is a symmetric algebra.

In this paper we develop a construction that extends widely the
semitrivial extension. Let $M$ be a bimodule over an algebra $R$.
For any $i\geq 2$, denote by $M^{\ot_R\, i}$ the $i$-fold tensor
product $M\ot_R M\ot_R\ldots\ot_RM$; also denote $M^{\ot_R\, 0}=R$
and $M^{\ot_R\, 1}=M$. The tensor algebra of the bimodule $M$ is
$T_R(M)=\oplus_{i\geq 0}M^{\ot_R\, i}$, with multiplication induced
by tensor concatenation. Let $n\geq 2$ be an integer, and let
$\varphi:M^{\ot_R\, n}\ra R$ be a morphism of $R$-bimodules which is
associative, in the sense that $\varphi(m_1\ot_R\ldots
\ot_Rm_n)m_{n+1}=m_1\varphi(m_2\ot_R\ldots\ot_Rm_{n+1})$ for any
$m_1,\ldots,m_{n+1}\in M$. We consider the factor algebra
$A(R,M,\varphi)$ of $T_R(M)$ by the two-sided ideal generated by all
elements of the form
$m_1\ot_R\ldots\ot_Rm_n-\varphi(m_1\ot_R\ldots\ot_Rm_n)$. We show
that $A(R,M,\varphi)$ can be identified with $R\oplus M
\oplus\ldots\oplus M^{\ot_R\, (n-1)}$, with multiplication induced
by tensor concatenation for tensor monomials whose sum of length is
less than $n$, respectively by tensor concatenation followed by
applying $\varphi$ to a (in fact to any) successive tensor sequence
of length $n$. It follows that $A(R,M,\varphi)$ has a
$\mathbb{Z}_n$-graded algebra structure.

We investigate Frobenius type properties of algebras obtained by
this construction. Frobenius algebras and their variants
(quasi-Frobenius algebras, symmetric algebras) form a very
interesting class of algebras, occurring in algebra, geometry and
topology. A finite-dimensional algebra $R$ is called quasi-Frobenius
if $R$ is an injective left $R$-module, and is called Frobenius if
$R^*$ and $R$ are isomorphic as left $R$-modules; if the isomorphism
is as bimodules, then $R$ is called symmetric. Graded versions of
these concepts exist, in particular an algebra $R$ graded by a group
$G$ is called $\sigma$-graded Frobenius, where $\sigma\in G$, if
$R^*$ is isomorphic as a graded left $R$-module with the shift
$R(\sigma)$ of $R$. Such objects occur in non-commutative geometry,
for example Koszul duals of certain connected Noetherian graded
algebras \cite{smith}, and certain algebras constructed from twisted
superpotentials, related to Calabi-Yau algebras \cite{HeXia}.

All algebras, spaces and linear duals will be over a field $K$. We prove the following.\\

{\bf Proposition A.} {\it Let $R$ be a finite-dimensional algebra,
$M$  a finite-dimensional $R$-bimodule, and let $\varphi:M^{\ot_R\,
n}\ra R$ be an associative morphism of
$R$-bimodules, where $n\geq 2$. The following assertions hold.\\
{\rm (i)} If $M$ is an invertible $R$-bimodule, then
$A(R,M,\varphi)$ is $\widehat{n-1}$-faithful. If moreover,
$M^{\ot_R\,(n-1)}\simeq R^*$ as left $R$-modules, then
$A(R,M,\varphi)$ is $\widehat{n-1}$-graded
Frobenius, in particular it is Frobenius.\\
{\rm (ii)} If  $M$ is invertible, then
$A(R,M,\varphi)$ is quasi-Frobenius if and only if $R$ is quasi-Frobenius.\\
{\rm (iii)} If $\varphi$ is an isomorphism, then the
$\mathbb{Z}_n$-graded
algebra $A(R,M,\varphi)$ is strongly graded.\\
{\rm (iv)} If $\varphi$ is an isomorphism and there exists $0\leq
i\leq n-1$ such that $M^{\ot_R\, i}\simeq R^*$ as left $R$-modules,
then $A(R,M,\varphi)$ is $\hat{i}$-graded Frobenius, in
particular it is Frobenius.}\\

We are interested in the special case where $M=R^*$. If $n\geq 2$
and $\varphi:(R^*)^{\ot_R\, n}\ra R$ is the zero morphism, we denote
$A(R,R^*,\varphi)$ by $A(R,n)$; the subclass consisting of such
algebras is of particular interest. An important fact for our
further considerations is that a finite-dimensional algebra $R$ is
quasi-Frobenius if and only if $R^*$ is an invertible $R$-bimodule.
In the case where $R$ is a Frobenius algebra and $n\geq 2$, we
determine all bimodule morphisms $\varphi:R^{\ot_R\, n}\ra R$ that
are associative; wee see that some bimodule morphisms may not be
associative. However,  even when $R$ is only quasi-Frobenius, if the
order of the class of $R^*$ in the Picard
group of $R$ is a divisor of $n$, which means that $(R^*)^{\ot_R\, n}\simeq R$ as $R$-bimodules, we prove the following.\\

{\bf Theorem B.} {\it Let $R$ be a finite-dimensional algebra such
that $(R^*)^{\ot_Rn}\simeq R$ as $R$-bimodules for some  $n\geq 2$
(note that $R$ is necessarily quasi-Frobenius).
Then any morphism of $R$-bimodules $\varphi:(R^*)^{\ot_Rn}\ra R$ is associative.}\\

We first prove Theorem B in the case where $R$ is Frobenius, by
investigating the Nakayama automorphism of $R$. In the case where
$R$ is quasi-Frobenius, we consider a basic algebra $S$ of $R$,
which is Frobenius and Morita equivalent to $R$. Then there is a
monoidal equivalence between the categories of bimodules over $R$
and $S$, which we use to transfer  the property to be proved in
Theorem B from $S$ to $R$.

In the situation where $\varphi:(R^*)^{\ot_Rn}\ra R$ is an
isomorphism of bimodules, thus associative by Theorem B, we show
that the strongly graded algebra $A(R,R^*,\varphi)$  enjoys good Frobenius properties.\\

{\bf Theorem C.} {\it Let $R$ be a finite-dimensional algebra such
that there is an isomorphism of $R$-bimodules
$\varphi:(R^*)^{\ot_R\, n}\ra R$ for some $n\geq 2$. Then\\
{\rm (i)} $A(R,R^*,\varphi)$ is $\hat{1}$-graded Frobenius, and it
is a
symmetric algebra.\\
{\rm (ii)} If moreover, $R$ is Frobenius, then $A(R,R^*,\varphi)$ is
$\hat{i}$-graded Frobenius for any $0\leq i<n$.}\\

We also investigate Frobenius properties of $A(R,R^*,\varphi)$ in
the case where $\varphi$ is not necessarily an isomorphism, provided that some additional properties are satisfied.\\

{\bf Theorem D.} {\it Let $R$ be a finite-dimensional algebra, and
let $\varphi:(R^*)^{\ot_R\, n}\ra R$ be an associative morphism of
bimodules, where $n\geq 2$. The following hold.\\
 {\rm (i)} If
$(R^*)^{\ot_R\, (n-2)}\simeq R$ as bimodules, then
$A(R,R^*,\varphi)$ is a symmetric algebra.\\
{\rm (ii)} If $R$ is a Frobenius algebra, then $A(R,R^*,\varphi)$ is
$\widehat{n-1}$-graded Frobenius, in particular it is a Frobenius
algebra. Moreover, if $\lambda$ is a Frobenius form on $R$, with
associated Nakayama automorphism $\nu$, then the linear map
$\Lambda:A(R,R^*,\varphi)\ra K$ such that
$$\Lambda((r_1\rhu \lambda)\ot_R\ldots\ot_R(r_{n-1}\rhu
\lambda))=\lambda(r_1\nu(r_2)\cdots\nu^{n-2}(r_{n-1}))$$ for any
$r_1,\ldots,r_{n-1}\in R$, and $\Lambda$ vanishes on any homogeneous
component of degree $\neq \widehat{n-1}$ of $A(R,R^*,\varphi)$, is a
Frobenius form on $A(R,R^*,\varphi)$, and the Nakayama automorphism
$\mathcal{N}$ of $A(R,R^*,\varphi)$ associated with $\Lambda$ is
given by
$$\mathcal{N}((r_1\rhu \lambda)\ot_R\ldots\ot_R(r_{i}\rhu
\lambda))=(\nu^{2-n}(r_1)\rhu
\lambda)\ot_R\ldots\ot_R(\nu^{2-n}(r_{i})\rhu \lambda)$$ for any
$1\leq i\leq n-1$, and any $r_1,\ldots,r_i\in R$, and
$\mathcal{N}(r)=\nu^{2-n}(r)$ for any $r\in R$.}\\

 On the other hand, If $R$ is
Frobenius, we use Theorem D (ii) for obtaining a criterion for
$A(R,R^*,\varphi)$ to be a symmetric algebra. As a consequence, we
show the following.\\

{\bf Corollary E.} {\it Let $R$ be a Frobenius algebra, $n\geq 2$,
and let $\varphi:(R^*)^{\ot_R\, n}\ra R$ be an associative morphism
of bimodules. Then

{\rm (i)} If $R$ is a local algebra, then $A(R,R^*,\varphi)$ is a
symmetric algebra if and only if either $(R^*)^{\ot_R\, (n-2)}\simeq
R$ as $R$-bimodules, or $\varphi$ is an isomorphism.

{\rm (ii)} $A(R,n)$ is a symmetric algebra
if and only if $(R^*)^{\ot_R\, (n-2)}\simeq R$ as $R$-bimodules. }\\

We give an example to show that the condition that $R$ is local
cannot be omitted in Corollary E (i).

If $R$ is quasi-Frobenius, then an algebra of type
$A(R,R^*,\varphi)$ is quasi-Frobenius by Proposition A (ii), however
we show that it might not be Frobenius. Let $R$ be a finite
dimensional quasi-Frobenius algebra and $e_1,\ldots, e_q$ be
idempotents of $R$ such that $Re_1,\ldots,Re_q$ is a system of
representatives for the isomorphism types of principal
indecomposable $R$-modules; let $m_1,\ldots,m_q$ be their
multiplicities in $R$. Then there is a permutation $\pi$ of
$1,\ldots,q$, called the Nakayama permutation, such that ${\rm
top}(Re_i)\simeq {\rm soc}(Re_{\pi(i)})$ for any $1\leq i\leq q$,
where for an $R$-module $M$ we denote by ${\rm soc}(M)$ its socle
and by ${\rm
top}(M)=M/J(R)M$. With these notations, we prove the following.\\

{\bf Theorem F.} {\it Let $R$ be  quasi-Frobenius and $n\geq 2$.
Then $A(R,n)$ is Frobenius if and only if $m_i=m_{\pi^{n-2}(i)}$ for any $1\leq i\leq q$. }\\

We mentioned that $A(R,2)$, the trivial extension of $R^*$ by $R$,
is a symmetric algebra for any finite-dimensional algebra $R$. We
give examples showing that if $n\geq 3$, $A(R,n)$ may be not even
quasi-Frobenius (for certain $R$ which is not quasi-Frobenius). On
the other hand, we give examples of $R$ which is not
quasi-Frobenius, while $A(n,R)$ is symmetric. Thus the case $n=2$ is
an exceptional situation.

 For basic facts and notation on Frobenius algebras we refer to
\cite{lam} and \cite{sy}, while for things about graded algebras to
\cite{nvo}.

\section{Graded Frobenius algebras and graded algebras from
bimodules}

 {\it (Quasi-)Frobenius algebras and symmetric algebras.}

Let $R$ be a finite-dimensional $K$-algebra. Denote by $\rhu$ and
$\lhu$ the usual left and right actions of $R$ on the dual space
$R^*$, which make $R^*$ an $R$-bimodule; thus $(r\rhu
r^*)(s)=r^*(sr)$ and $(r^*\lhu r)(s)=r^*(rs)$ for any $r^*\in R^*$
and $r,s\in R$. $R$ is called Frobenius if $R^*$ is isomorphic to
$R$ as a left $R$-module (or equivalently, as a right $R$-module).
If $R^*$ is isomorphic to $R$ as an $R$-bimodule, then $R$ is called
symmetric.

A basic fact is that $R$ is Frobenius if and only if there exists
$\lambda\in R^*$ such that $R\rhu \lambda=R^*$; such a $\lambda$ is
called a Frobenius form on $R$, and it also satisfies $\lambda \lhu
R=R^*$. The Nakayama automorphism $\nu:R\ra R$ associated with
$\lambda$ is the unique map such that $\lambda \lhu r=\nu(r)\rhu
\lambda$ for any $r\in R$; $\nu$ is an algebra automorphism, and
$R^*$ is isomorphic as an $R$-bimodule to $_1R_\nu$, where for an
$R$-bimodule $M$ and two automorphisms $\alpha$ and $\beta$ of $R$,
we denote by $_\alpha M_\beta$ the $R$-bimodule with underlying
space $M$, left action $r*m=\alpha(r)m$, and right action
$m*r=m\beta(r)$.

An algebra $R$ is symmetric if and only if there is a Frobenius form
$\lambda$ such that $\lambda (rs)=\lambda (sr)$ for any $r,s\in R$;
in this situation, the associated Nakayama automorphism is inner.

A finite-dimensional algebra $R$ is called quasi-Frobenius if it is
self-injective. By \cite[Corollary 2.3]{dnn2}, this is equivalent to
$R^*$ being an invertible $R$-bimodule. We recall that an
$R$-bimodule is called invertible
if there is a bimodule $N$ such that $M\ot_RN\simeq N\ot_RM\simeq R$ as $R$-bimodules. \\

 {\it Graded algebras and $\sigma$-graded Frobenius
conditions.}

Let $G$ be a group with identity element $e$. A $G$-graded algebra
is an algebra $A$ with a linear decomposition $A=\oplus_{g\in
G}A_g$, such that $A_gA_h\subset A_{gh}$ for any $g,h\in G$. If
$\sigma\in G$, $A$ is called $\sigma$-faithful if for any $g\in G$
and any $a_g\in A_g$ such that $A_{\sigma g^{-1}}a_g=0$, we have
that $a_g=0$.

A $G$-graded algebra $A$ is strongly graded if $A_gA_h=A_{gh}$ for
any $g,h\in G$, or equivalently, if $1\in A_gA_{g^{-1}}$ for any
$g\in G$. A strongly graded algebra is $\sigma$-faithful for any
$\sigma\in G$.

If $A$ is a finite-dimensional $G$-graded algebra, then $A^*$ is a
graded  $A$-bimodule, with the grading given by
$$(A^*)_g=\{ a^*\in A^*\;|\; a^*(A_h)=0 \mbox{ for any }h\neq
g^{-1}\}.$$ If $\sigma\in G$, denote by $A(\sigma)$ the graded left
$A$-module which is just $A$ as a left $A$-module, and has a grading
shifted by $\sigma$, i.e., $A(\sigma)_g=A_{g\sigma}$ for any $g\in
G$. Then $A$ is called $\sigma$-graded Frobenius if $A^*\simeq
A(\sigma)$ as graded left $A$-modules. Clearly, if $A$ is
$\sigma$-graded Frobenius, then $A$ is a Frobenius algebra. We will
need the following characterization of $\sigma$-graded Frobenius
algebras.

\begin{proposition}
A finite-dimensional graded algebra $A$ is $\sigma$-graded Frobenius
if and only if it has a Frobenius form $\lambda$ of degree
$\sigma^{-1}$ in $A^*$.
\end{proposition}
\begin{proof}
If $\varphi:A(\sigma)\ra A^*$ is an isomorphism of graded left
$A$-modules, then $\lambda=\varphi(1)$ is a Frobenius form of degree
$\sigma^{-1}$. Conversely, if $\lambda$ is such a form, then the map
$\varphi:A(\sigma)\ra A^*$, $\varphi(r)=r\rhu \lambda$, is an
isomorphism of graded left $A$-modules.
\end{proof}

Note that if $A$ is $\sigma$-graded Frobenius and $\sigma$ lies in
the centre of $G$, then the associated Nakayama automorphism $\nu$
preserves degree.

 We recall from \cite[Theorem 4.3]{dnn1} that a
finite-dimensional $G$-graded algebra $A$ is $\sigma$-graded
Frobenius if and only if it is $\sigma$-faithful and $A_\sigma\simeq
A_e^*$ as left $A_e$-modules. For such an $A$, a Frobenius form
$\lambda\in A^*$ of degree $\sigma^{-1}$ can be obtained by taking
an isomorphism $\theta:A_\sigma\ra A_e^*$ of left $A_e$-modules,
and defining $\lambda(x)=\theta(x)(1)$ for any $x\in A_\sigma$, and $\lambda(x)=0$ for any homogeneous $x$ of degree $\neq \sigma$. \\

 {\it Graded algebras from bimodules}

Let $R$ be an algebra and let $M$ be an $R$-bimodule.  The tensor
algebra of the bimodule $M$ is the space $T_R(M)=\oplus_{n\geq
0}M^{\ot_R\, n}$, with multiplication (which we occasionally denote
by $\cdot$) induced by tensor concatenation and the left and right
actions of $R$ on $M$; this is a positively graded algebra with
homogeneous component of degree $n$ being $T_R(M)_n=M^{\ot_R\, n}$
for any $n\geq 0$.

If $n\geq 2$, a bimodule morphism $\varphi:M^{\ot_R\, n}\ra R$ is
called associative if
\begin{equation} \label{defassociative}
\varphi(m_1\ot_R\ldots
\ot_Rm_n)m_{n+1}=m_1\varphi(m_2\ot_R\ldots\ot_Rm_{n+1})
\end{equation} for any $m_1,\ldots,m_{n+1}\in M$. If $\varphi$ is associative, it can be easily
proved by induction on $p$ that for any $r\geq n+1$, any $1\leq
p\leq r-n$ and any $m_1,\ldots,m_r\in M$ we have
\begin{equation} \label{generalizedassoc}
\begin{aligned}
 m_1\ot_R\ldots
\ot_Rm_{p-1}\ot_R\varphi(m_p\ot_R\ldots\ot_Rm_{p+n-1})m_{p+n}\ot_Rm_{p+n+1}\ot_R\ldots
m_r=\\
\varphi(m_1\ot_R\ldots \ot_Rm_n)m_{n+1}\ot m_{n+2}\ot_R\ldots
\ot_Rm_r.
\end{aligned}
\end{equation}

If $\varphi:M^{\ot_Rn}\ra R$ is associative, then we have a Morita
context $(R,R,M,M^{\ot_R(n-1)},[,],(,))$ connecting the rings $R$
and $R$, where both Morita maps $[,]$ and $(,)$ are induced by
$\varphi$ in the obvious way. As a consequence, if $\varphi$ is
surjective, then it is an isomorphism of $R$-bimodules, in
particular $M$ is an invertible bimodule.

Now let $M$ be an $R$-bimodule, and let $\varphi:M^{\ot_R\, n}\ra R$
be an associative morphism of bimodules, where $n\geq 2$. Let $I$ be
the two-sided ideal of $T_R(M)$ generated by all elements of the
form $m_1\ot_R\ldots \ot_Rm_n-\varphi(m_1\ot_R\ldots\ot_Rm_n)$, with
$m_1,\ldots,m_n\in M$. We consider the factor algebra $T_R(M)/I$,
which we denote by $A(R,M,\varphi)$. In order to describe this
factor algebra, we need the following.

\begin{lemma}
{\rm (i)} $I$ is the left ideal generated by all elements of the
form $m_1\ot_R\ldots \ot_Rm_n-\varphi(m_1\ot_R\ldots\ot_Rm_n)$, with
$m_1,\ldots,m_n\in M$.\\
{\rm (ii)} $I\cap (T_R(M)_0+T_R(M)_1+\ldots +T_R(M)_{n-1})=0$.\\
{\rm (iii)} For any $u\in T_R(M)$ there exists $u'\in I$ such that
$u-u'\in T_R(M)_0+T_R(M)_1+\ldots +T_R(M)_{n-1}$.
\end{lemma}
\begin{proof}
(i) Let $z=m_1\ot_R\ldots \ot_Rm_n-\varphi(m_1\ot_R\ldots\ot_Rm_n)$
and $m\in M$. The statement follows from
 \bea z\cdot m&=&
m_1\ot_R\ldots \ot_Rm_n\ot_R m -\varphi(m_1\ot_R\ldots\ot_Rm_n)m\\
&=&m_1\cdot (m_2\ot_R\ldots \ot_Rm_n\ot_R m)
-m_1\varphi(m_2\ot_R\ldots\ot_Rm_n\ot_Rm)\\
&=&m_1\cdot (m_2\ot_R\ldots \ot_Rm_n\ot_R m
-\varphi(m_2\ot_R\ldots\ot_Rm_n\ot_Rm) ) \eea and the fact that
$T_R(M)$ is generated by $T_R(M)_1=M$.

(ii) Let $$y=\sum_{j\in J}z_j\cdot
(m_{j1}\ot_R\ldots\ot_Rm_{jn}-\varphi
(m_{j1}\ot_R\ldots\ot_Rm_{jn}))\in T_R(M)_0+T_R(M)_1+\ldots
+T_R(M)_{n-1},$$ where the $m_{ji}$'s lie in $M$ and $(z_j)_{j\in
J}$ is a family of elements of $T_R(M)$; we may assume that each
$z_j$ is a homogeneous element. For any $d$, let $J_d$ be the set of
all $j$ such that $z_j$ has degree $d$, and denote $y_d=\sum_{j\in
J_d}z_j\cdot (m_{j1}\ot_R\ldots\ot_Rm_{jn}-\varphi
(m_{j1}\ot_R\ldots\ot_Rm_{jn}))$. Clearly, $y=\sum_dy_d$. If $y\neq
0$, let $d_0$ be the greatest integer with $y_{d_0}\neq 0$. Then
$$y=\sum_{d\leq d_0}y_d=\sum_{j\in J'}z_j\cdot
(m_{j1}\ot_R\ldots\ot_Rm_{jn})-\sum_{j\in J'}z_j\varphi
(m_{j1}\ot_R\ldots\ot_Rm_{jn}),$$ where $J'=\cup_{d\leq d_0}J_d$,
and in this sum, the highest degree terms are $z_j\cdot
(m_{j1}\ot_R\ldots\ot_Rm_{jn})$ with $j\in J_{d_0}$, and they have
degree $d_0+n$. As any element in $T_R(M)_0+T_R(M)_1+\ldots
+T_R(M)_{n-1}$ has all non-zero homogeneous components of degrees at
most $n-1$, we obtain that $\sum_{j\in J_{d_0}}z_j\cdot
(m_{j1}\ot_R\ldots\ot_Rm_{jn})=0$. Applying ${\rm
Id}^{\ot_R\,d}\ot_R\varphi$, we get $\sum_{j\in J_{d_0}}z_j\varphi
(m_{j1}\ot_R\ldots\ot_Rm_{jn})=0$, so then $$y_{d_0}=\sum_{j\in
J_{d_0}}z_j\cdot (m_{j1}\ot_R\ldots\ot_Rm_{jn})-\sum_{j\in
J_{d_0}}z_j\varphi (m_{j1}\ot_R\ldots\ot_Rm_{jn})=0,$$ a
contradiction. We conclude that $y$ must be zero.

(iii) We proceed by induction on the highest degree $d$ of a
non-zero homogeneous component $u_d$ of $u$. It is obvious if $d<n$.
For the induction step, if $d\geq n$ and $u_d=\sum_j
m_{j1}\ot_R\ldots \ot_Rm_{jd}$, let us take $$z=\sum_j
(m_{j1}\ot_R\ldots \ot_Rm_{j,d-n})\cdot (m_{j,d-n+1}\ot_R\ldots
\ot_Rm_{j,d}-\varphi(m_{j,d-n+1}\ot_R\ldots \ot_Rm_{j,d})),$$ which
lies in $I$, and apply the induction hypothesis to $u-z$.
\end{proof}

As a consequence of the previous lemma, we can identify
$$A(R,M,\varphi)=R\oplus M\oplus (M\ot_RM)\oplus\ldots\oplus
M^{\ot_R\, (n-1)},$$ and the multiplication is such that for any
$1\leq i,j\leq n-1$ and $m_1,\ldots,m_i,p_1,\ldots,p_j\in M$,
$$(m_1\ot_R\ldots\ot_Rm_i)\cdot(p_1\ot_R\ldots\ot_Rp_j)=m_1\ot_R\ldots\ot_Rm_i\ot_Rp_1\ot_R\ldots
\ot_Rp_j$$ if $i+j<n$, and
$$(m_1\ot_R\ldots\ot_Rm_i)\cdot(p_1\ot_R\ldots\ot_Rp_j)=\varphi(m_1\ot_R\ldots
\ot_Rm_i\ot_Rp_1\ot_R\ldots\ot_Rp_{n-i})p_{n-i+1}\ot_R\ldots\ot_Rp_j$$
whenever $i+j\geq n$; the elements of $R$ act on each component
according to the $R$-bimodule structure. In other words, the
multiplication is the tensor concatenation, followed by applying
$\varphi$ on (any) $n$ successive tensor position if the length is
at least $n$.

Clearly, $A(R,M,\varphi)$ is a $\mathbb{Z}_n$-graded algebra, whose
homogeneous component of degree $\hat{i}$ is $M^{\ot_Ri}$ for any
$0\leq i\leq n-1$; here we denote by $\hat{i}$ the class of $i$
modulo $n$, and $\mathbb{Z}_n$ is the additive group of integers
modulo $n$. If both $R$ and $M$ have finite dimension, then
$A(R,M,\varphi)$ is a finite-dimensional algebra. We note that
$A(R,M,\varphi)_{\hat{i}}A(R,M,\varphi)_{\hat{j}}=A(R,M,\varphi)_{\widehat{i+j}}$
for any $0\leq i,j$ with $i+j<n$, and
$A(R,M,\varphi)_{\hat{i}}A(R,M,\varphi)_{\widehat{n-i}}={\rm Im}(\varphi)$ for any $0<i<n$. \\

{\bf Proof of Proposition A.} Denote for simplicity $A(R,M,\varphi)$ by $A$.\\
(i) Let $z\in A_{\hat{i}}$ be a homogeneous element, where $0\leq
i\leq n-1$, and assume that $A_{\widehat{n-1-i}}z=0$. Thus
$M^{\ot_R\,(n-i-1)}\ot_Rz=0$ (inside
$M^{\ot_R\,(n-i-1)}\ot_RM^{\ot_R\, i}$). Since $M$ is invertible as
a bimodule, then so is $M^{\ot_R\,(n-i-1)}$, and denote by $V$ its
inverse. We then have $V\ot_RM^{\ot_R\,(n-i-1)}\ot_Rz=0$, and using
the isomorphism $V\ot_RM^{\ot_R\,(n-i-1)}\ot_RM^{\ot_R\, i}\simeq
M^{\ot_R\, i}$, we get that $z=0$.

If $M^{\ot_R\,(n-1)}\simeq R^*$ as left $R$-modules, then $A$ is
$\widehat{n-1}$-graded Frobenius by
\cite[Theorem 4.3]{dnn1}.\\
(ii) Assume that $R$ is quasi-Frobenius. Recall from \cite[Section
2.5]{nvo} that the coinduced functor ${\rm Coind}$ associates to any
$R$-module $N$ the graded $A$-module ${\rm Coind}(N)={\rm
Hom}_R(A,N)$, with the grading
$${\rm Coind}(N)_{\hat{j}}=\{ f\in {\rm Hom}_R(A,N)|\,
f(A_{\hat{i}})=0\mbox{ for any }\hat{i}\neq \widehat{n-j}\}.$$ There
is a natural morphism of graded left $A$-modules $$F:A\ra {\rm
Coind}(A_{\widehat{n-1}})(\hat{1}),
F(a_{\hat{i}})(b)=(ba_{\hat{i}})_{\widehat{n-1}},$$ for any
$\hat{i}\in \mathbb{Z}_n, a_{\hat{i}}\in A_{\hat{i}}$ and $b\in A$.
By (i), $A$ is $\widehat{n-1}$-faithful, and then $F$ is injective
by \cite[Proposition 2.6.3 (e)]{nvo}.

Since $M$ is an invertible $R$-bimodule, then so is $M^{\ot_R\, j}$
for any $1\leq j\leq n-1$. As an invertible bimodule, $M^{\ot_R\,
(n-1)}$ is a projective left $R$-module, and since $R$ is
quasi-Frobenius, $A_{\widehat{n-1}}=M^{\ot_R\, (n-1)}$ is also an
injective left $R$-module, see \cite[Theorem 15.9]{lam}. By
\cite[Proposition 2.6.3 (e)]{nvo}, ${\rm Coind}(A_{\widehat{n-1}})$
is an injective object in the category of graded left $A$-modules,
thus also injective as a left $A$-module by \cite[Corollary
2.5.2]{nvo}. We show that $F$ is an isomorphism, and then $A$ is an
injective left $A$-module, too, which means that $A$ is
quasi-Frobenius.  For this aim, it is enough to check that the
homogeneous components of the same degree of $A$ and ${\rm
Coind}(A_{\widehat{n-1}})(\hat{1})$ have the same dimension. For any
$0\leq j\leq n-1$,
 $${\rm
Coind}(A_{\widehat{n-1}})(\hat{1})_{\hat{j}}={\rm
Coind}(A_{\widehat{n-1}})_{\widehat{j+1}}\simeq {\rm
Hom}_R(A_{\widehat{n-j-1}},A_{\widehat{n-1}})={\rm Hom}_R(M^{\ot_R\,
(n-j-1)},M^{\ot_R\, (n-1)}).$$ As $M^{\ot_R\, (n-j-1)}$ is an
invertible $R$-bimodule, the functor $M^{\ot_R\,
(n-j-1)}\ot_R(-):R-{\rm mod}\ra R-{\rm mod}$ is a $K$-linear
equivalence of categories, so the spaces ${\rm Hom}_R(M^{\ot_R\,
(n-j-1)},M^{\ot_R\, (n-1)})$ and ${\rm Hom}_R(R,M^{\ot_R\, j})$ are
isomorphic. Then ${\rm dim}({\rm Hom}_R(M^{\ot_R\,
(n-j-1)},M^{\ot_R\, (n-1)}))={\rm dim}(M^{\ot_R\, j})={\rm
dim}(A_{\hat{j}})$, showing that $F$ is indeed an isomorphism.

Conversely, assume that $A$ is quasi-Frobenius. Since each
$M^{\ot_R\, j}$ with $0\leq j\leq n-1$ is a projective right
$R$-module, we see that $A$ is also projective as a right
$R$-module, and then the functor $A\ot_R (-):R-{\rm mod}\ra A-{\rm
mod}$ is exact, and then its right adjoint, the restriction of
scalars from $A-{\rm mod}$ to $R-{\rm mod}$ takes injectives to
injectives. As $A$ is $A$-injective, it is also $R$-injective,
and then so is its direct summand (as a left $R$-module) $A_{\hat{0}}=R$. Thus $R$ is quasi-Frobenius.  \\
(iii) We have that $A_{\hat{i}}A_{\widehat{n-i}}={\rm
Im}(\varphi)=R=A_{\hat{0}}$ for any $1\leq i\leq n-1$, and obviously
$A_{\hat{0}}A_{\hat{0}}=A_{\hat{0}}$, thus $A$ is strongly graded. \\
(iv) Since $\varphi$ is an isomorphism, then $A$ is strongly graded,
thus $\hat{i}$-faithful for any $\hat{i}$. Since
$A_{\hat{i}}=M^{\ot_R\, i}\simeq R^*$ as left $R$-modules, then $A$
is $\hat{i}$-graded Frobenius, again by \cite[Theorem 4.3]{dnn1}.
\qed

\section{Associative bimodule morphisms} \label{sectionassociative}

If $R$ is an algebra and $\alpha,\beta$ are algebra automorphisms of
$R$, then there is a bimodule isomorphism $F:{_1R_\alpha}\ot
{_1R_\beta}\ra {_1R_{\alpha\beta}}$ given by
$F(r\ot_Rs)=r\alpha(s)$; its inverse is $F^{-1}(r)=r\ot_R1$.
Extending this, if $p\geq 2$ and $\alpha_1,\ldots,\alpha_p$ are
automorphisms of $R$, there is an isomorphism of bimodules
$$F:{_1R_{\alpha_1}}\ot_R{\;_1R_{\alpha_2}}\ot_R\ldots\ot_R{\;_1R_{\alpha_p}}\ra
{_1R_{\alpha_1\ldots\alpha_p}},$$ $$F(r_1\ot_Rr_2\ot_R\ldots
\ot_Rr_p)=r_1\alpha_1(r_2)\cdots
(\alpha_1\ldots\alpha_{p-1})(r_p),$$ whose inverse is given by
$F^{-1}(r)=r\ot_R1\ot_R\ldots \ot_R1$.

In the case where $R$ is a Frobenius algebra with a Frobenius form
$\lambda$ and associated Nakayama automorphism $\nu$, since there is
an $R$-bimodule isomorphism ${_1R_\nu}\ra R^*$ which takes $r\in R$
to $r\rhu \lambda$, we obtain a bimodule isomorphism
$$\theta:(R^*)^{\ot_R\, p}\ra {_1R_{\nu^p}}, \theta ((r_1\rhu
\lambda)\ot_R\ldots\ot_R(r_p\rhu
\lambda))=r_1\nu(r_2)\cdots\nu^{p-1}(r_p).$$

On the other hand, if $\alpha$ is an automorphism of $R$, then a
bimodule morphism $\gamma:{_1R_\alpha}\ra R$ is of the form
$\gamma(r)=rc$, where $c$ is an element of $R$ such that
$\alpha(r)c=cr$ for any $r\in R$. Combining with the considerations
above, if $R$ is Frobenius and $p\geq 2$, we obtain that a bimodule
morphism $\varphi:(R^*)^{\ot_R\, p}\ra R$ is of the form
$\varphi((r_1\rhu \lambda)\ot_R\ldots\ot_R(r_p\rhu
\lambda))=r_1\nu(r_2)\cdots\nu^{p-1}(r_p)c$, where $c\in R$ such
that $\nu^p(r)c=cr$ for any $r\in R$. In order to determine when is
such a $\varphi$ associative, we see that  $$\varphi((r_1\rhu
\lambda)\ot_R\ldots\ot_R(r_p\rhu \lambda))\rhu (r_{p+1}\rhu
\lambda)=(r_1\nu(r_2)\cdots\nu^{p-1}(r_p)cr_{p+1})\rhu \lambda,$$
and \bea (r_1\rhu \lambda)\lhu \varphi((r_2\rhu
\lambda)\ot_R\ldots\ot_R(r_{p+1}\rhu \lambda))&=&(r_1\rhu
\lambda)\lhu (r_2\nu(r_3)\cdots\nu^{p-1}(r_{p+1})c)\\
&=&(r_1\nu(r_2\nu(r_3)\cdots\nu^{p-1}(r_{p+1})c))\rhu \lambda\\
&=&(r_1\nu(r_2)\cdots \nu^{p-1}(r_p)\nu^p(r_{p+1})\nu(c))\rhu
\lambda,\eea so then $\varphi$ is associative if and only if
$scr=s\nu^p(r)\nu(c)$ for any $r,s\in R$. Since $cr=\nu^p(r)c$, it
is further equivalent to $s\nu^p(r)c=s\nu^p(r)\nu(c)$ for any $r,s$,
thus to $\nu(c)=c$.

We conclude that the associative bimodule morphisms $(R^*)^{\ot_R\,
p}\ra R$ are in bijection to the elements $c\in R$ satisfying
$\nu(c)=c$ and $\nu^p(r)c=cr$ for any $r\in R$. \\

\begin{example}
{\rm If $q\in K, q\neq 0$, let $K_q[X,Y]=K\{ X,Y\}/(YX-qXY)$ be the
quantum plane associated with $q$,  and let
$R_q=K_q[X,Y]/(X^2,Y^2)$. Then $R_q$  has  a basis $\{ 1,x,y,xy\}$,
where $x,y$ denote the classes of $X,Y$ in $R$, with relations
$x^2=y^2=0$ and $yx=qxy$. If $\{ 1^*,x^*,y^*,(xy)^*\}$ is the dual
basis of $R_q^*$, then $(xy)^*$ is a Frobenius form on $R_q$, thus
$R_q$ is a Frobenius algebra. The associated Nakayama automorphism
$\nu$ is given by $\nu(x)=q^{-1}x, \nu(y)=qy$.

Let $n\geq 2$. Associative morphisms of $R_q$-bimodules
$(R_q^*)^{\ot_{R_q}\, n}\ra R_q$ are in bijection to the elements
$c\in R_q$ such that $\nu(c)=c$, $q^{-n}xc=cx$ nd $q^nyc=cy$. If we
write $c=\alpha+\beta x+\gamma y+\delta xy$, with
$\alpha,\beta,\gamma,\delta\in K$, the first of the three relations
is equivalent to $\beta=\gamma=0$, while the other two are
equivalent to $q^n\alpha=\alpha$.

We conclude that if $q$ is not a $n$th root of unity, then $c=\delta
xy$, and any associative bimodule morphisms $(R_q^*)^{\ot_{R_q}\,
n}\ra R_q$ is a scalar multiple of the one obtained for $c=xy$,
while if $q^n=1$, then each $c=\alpha +\delta y$, with
$\alpha,\delta \in K$, determines an associtive bimodule morphism,
which is an isomorphism if and only if $\alpha\neq 0$.}
\end{example}

The rest of this section is devoted to proving Theorem B. Thus let
$R$ be a finite-dimensional algebra such that $(R^*)^{\ot_R\,
n}\simeq R$ as $R$-bimodules. In other words, $R^*$ is an invertible
$R$-bimodule and the order of its class in the Picard
group ${\rm Pic}(R)$ divides $n$. Since $R^*$ is invertible, we must have that $R$ is quasi-Frobenius.\\

{\it The case where $R$ is Frobenius.}

We first discuss the situation where $R$ is a Frobenius algebra. We
keep the notations $\lambda,\nu$ as above. Let
$\varphi_0:(R^*)^{\ot_R\, n}\ra R$ be a bimodule isomorphism. As we
discussed above, there is an invertible element $u\in R$ such that
$$\varphi_0((r_1\rhu
\lambda)\ot_R\ldots \ot_R(r_n\rhu \lambda))=r_1\nu(r_2)\cdots
\nu^{n-1}(r_n)u,$$ and $\nu^n(r)u=ur$, i.e., $\nu^n(r)=uru^{-1}$ for
any $r\in R$.

\begin{proposition} \label{lemanuu}
$\nu(u)=u$, thus $\varphi_0$ is associative.
\end{proposition}
\begin{proof}
Since $\lambda\lhu r=\nu(r)\rhu \lambda$, we have that
$\lambda(rx)=\lambda(x\nu(r))$ for any $r,x\in R$. For $x=1$, this
shows that $\lambda(r)=\lambda(\nu(r))$, thus $\lambda=\lambda\nu$,
and then $\lambda=\lambda\nu^i$ for any $i\geq 1$. Now \bea
\lambda(ru^{-1})&=&\lambda (u^{-1}\nu(r))\\
&=&\lambda(u^{-1}\nu^n(\nu^{-n+1}(r)))\\
&=&\lambda(u^{-1}u\nu^{-n+1}(r)u^{-1})\\
&=&\lambda(\nu^{-n+1}(r)u^{-1})\\
&=&\lambda\nu^{n-1}(\nu^{-n+1}(r)u^{-1})\\
&=&\lambda(r\nu^{n-1}(u^{-1})), \eea thus $u^{-1}\rhu
\lambda=\nu^{n-1}(u^{-1})\rhu \lambda$, showing that
$\nu^{n-1}(u^{-1})=u^{-1}$. Then
$\nu(u^{-1})=\nu^n(u^{-1})=uu^{-1}u^{-1}=u^{-1}$, so also
$\nu(u)=u$.
\end{proof}

\begin{corollary} \label{asocFrobenius}
If $R$ is a Frobenius algebra and $(R^*)^{\ot_R\, n}\simeq R$ as
bimodules, then any morphism of bimodules $\varphi:(R^*)^{\ot_R\,
n}\ra R$  is associative.
\end{corollary}
\begin{proof}
$\varphi \varphi_0^{-1}:R\ra R$ is a bimodule morphism, thus there
is a central element $a\in R$ such that $\varphi
\varphi_0^{-1}(r)=ar$ for any $r\in R$. Then
$\varphi(r)=a\varphi_0(r)$, and the associativity of $\varphi$
follows directly from the one of $\varphi_0$ and the fact that
$a\rhu r^*=r^*\lhu a$ for any $r^*\in R^*$.
\end{proof}

{\it The case where $R$ is quasi-Frobenius.}

Now we consider the general case, where $R$ is quasi-Frobenius. Let
$S$ be a basic algebra of $R$, which is Frobenius, see \cite[pages
172 and 351]{sy}. The algebras $S$ and $R$ are Morita equivalent,
and let
$$(S,R,_SP_R,_RQ_S,[\,,]:P\ot_RQ\ra S,(\, ,):Q\ot_SP\ra R)$$
be a strict Morita context connecting $S$ and $R$. As noticed in
\cite[pages 301-302]{takeuchi}, there is also an equivalence between
the categories of $S$-bimodules and $R$-bimodules, given by the
functor $F=Q\ot_S\cdot\ot_SP$.

If $M$ is an $S$-bimodule, let $\theta_M:F(M)^{\ot_Rn}\ra
F(M^{\ot_Sn})$ be the natural isomorphism, which works as
$$\theta_M(q_1\ot_Sm_1\ot_Sp_1\ot_R\ldots\ot_Rq_n\ot_Sm_n\ot_Sp_n)=q_1\ot_Sm_1[p_1,q_2]\ot_S\ldots\ot_Sm_{n-1}[p_{n-1},q_n]\ot_Sm_n\ot_S
p_n.$$ Also, let $\gamma:F(S)\ra R$ be the isomorphism of bimodules
defined by $\gamma(q\ot_Ss\ot_Sp)=(qs,p)$.

Now if $\psi:M^{\ot_Sn}\ra S$ is a morphism of $S$-bimodules, let
$\tilde{\psi}=\gamma F(\psi)\theta_M:F(M)^{\ot_Rn}\ra R$, which is a
morphism of $R$-bimodules. Explicitly,
$$\tilde{\psi}(q_1\ot_Sm_1\ot_Sp_1\ot_R\ldots\ot_Rq_n\ot_Sm_n\ot_Sp_n)=(q_1\psi(m_1[p_1,q_2]\ot_S\ldots\ot_Sm_{n-1}[p_{n-1},q_n]\ot_Sm_n),p_n).$$

\begin{proposition} \label{psitildeasociativ}
Let $M$ be an $S$-bimodule, and let $\psi:M^{\ot_Sn}\ra S$ be an
associative morphism of $S$-bimodules. Then
$\tilde{\psi}:F(M)^{\ot_Rn}\ra R$ is an associative morphism of
$R$-bimodules.
\end{proposition}
\begin{proof}
Let $z_i=q_i\ot_Sm_i\ot_Sp_i$ for $1\leq i\leq n+1$. Then

\begin{flalign*}
&\tilde{\psi}(z_1\ot_R\ldots\ot_Rz_n)z_{n+1}=&\\
&=(q_1\psi(m_1[p_1,q_2]\ot_S\ldots\ot_Sm_{n-1}[p_{n-1},q_n]\ot_Sm_n),p_n)q_{n+1}\ot_Sm_{n+1}\ot_Sp_{n+1}&\\
&=
q_1[\psi(m_1[p_1,q_2]\ot_S\ldots\ot_Sm_{n-1}[p_{n-1},q_n]\ot_Sm_n)p_n,q_{n+1}]\ot_Sm_{n+1}\ot_Sp_{n+1}&\\
&=q_1\ot_S\psi(m_1[p_1,q_2]\ot_S\ldots\ot_Sm_{n-1}[p_{n-1},q_n]\ot_Sm_n)[p_n,q_{n+1}]m_{n+1}\ot_Sp_{n+1}&\\
&=
q_1\ot_Sm_1[p_1,q_2]\psi(m_2[p_2,q_3]\ot_S\ldots\ot_Sm_{n-1}[p_{n-1},q_n]\ot_Sm_n\ot_S[p_n,q_{n+1}]m_{n+1})\ot_Sp_{n+1}&\\
&=q_1\ot_Sm_1\ot_S[p_1,q_2]\psi(m_2[p_2,q_3]\ot_S\ldots\ot_Sm_{n-1}[p_{n-1},q_n]\ot_Sm_n\ot_S[p_n,q_{n+1}]m_{n+1})p_{n+1}&\\
&=q_1\ot_Sm_1\ot_Sp_1(q_2,\psi(m_2[p_2,q_3]\ot_S\ldots\ot_Sm_{n-1}[p_{n-1},q_n]\ot_Sm_n\ot_S[p_n,q_{n+1}]m_{n+1})p_{n+1})&\\
&=z_1\tilde{\psi}(z_2\ot_R\ldots\ot_Rz_{n+1}), &
\end{flalign*}
thus $\tilde{\psi}$ is associative.
\end{proof}

\begin{corollary}\label{transferasoc}
Let $M$ be an $S$-bimodule such that any morphism of $S$-bimodules
$\psi:M^{\ot_Sn}\ra S$ is associative. Then any morphism of
$R$-bimodules $\varphi:F(M)^{\ot_Rn}\ra R$ is associative.
\end{corollary}
\begin{proof}
Let $\varphi:F(M)^{\ot_Rn}\ra R$ be a morphism of $R$-bimodules.
Since $F$ is an equivalence of categories, it is full, so there
exists a morphism of $S$-bimodules $\psi:M^{\ot_Sn}\ra S$ such that
$F(\psi)=\gamma^{-1}\varphi\theta_M^{-1}$. As $\psi$ is associative,
then $\tilde{\psi}$ is associative by Proposition
\ref{psitildeasociativ}. But $\tilde{\psi}=\gamma
F(\psi)\theta_M=\varphi$, and we are done.
\end{proof}

To finalize the proof of Theorem B, we recall from \cite[Proposition
3.2]{dnn2} that $F(S^*)\simeq R^*$ as $R$-bimodules. As $S$ is
Frobenius, we know by the Frobenius case, discussed earlier, that
any morphism of $S$-bimodules $\psi:(S^*)^{\ot_Sn}\ra S$ is
associative. Using Corollary \ref{transferasoc}, we conclude that
any morphism of $R$-bimodules $\varphi:(R^*)^{\ot_Rn}\ra R$ is
associative.

\section{Frobenius and symmetric properties of $A(R,R^*,\varphi)$}

We first consider the case where $A(R,R^*,\varphi)$ is a strongly
graded algebra. This is equivalent to $\varphi$ being an
isomorphism, in particular $R^*$ is an invertible $R$-bimodule, so
$R$ is quasi-Frobenius.\\

{\bf Proof of Theorem C.} Denote for simplicity
$A=A(R,R^*,\varphi)$.  Since $\varphi$ is an isomorphism, $A$ is
strongly graded, so it is $\hat{i}$-faithful for any $\hat{i}\in
\mathbb{Z}_n$.

(i) As $A$ is $\hat{1}$-faithful, and $A_{\hat{1}}=R^*\simeq
A_{\hat{0}}^*$ as left $R$-modules, we get that $A$ is
$\hat{1}$-graded Frobenius. Since an isomorphism
$\theta:A_{\hat{1}}\ra A_{\hat{0}}^*$ is just the identity map, we
can associate to $\theta$ a Frobenius form on $A$ defined by
$\Lambda:A\ra K$, $\Lambda (r^*)=r^*(1)$ for any $r^*\in
A_{\hat{1}}$, and $\Lambda (a)=0$ for any homogeneous $a\in A$ of
degree $\neq \hat{1}$. We have seen that $\Lambda$ is a Frobenius
form making $A$ into a $\hat{1}$-graded Frobenius algebra. We show
that $\Lambda (ab)=\Lambda (ba)$ for any homogeneous elements
$a,b\in A$, and this will show that $A$ is a symmetric algebra. It
is enough to check this for $a,b$ such that ${\rm deg}(ab)=\hat{1}$.

We first note that $\Lambda (r\rhu r^*)=(r\rhu
r^*)(1)=r^*(r)=(r^*\lhu r)(1)=\Lambda (r^*\lhu r)$ for any $r\in R$
and $r^*\in R^*$. This shows that if ${\rm deg}(a)=\hat{0}$ and
${\rm deg}(b)=\hat{1}$, then $\Lambda (ab)=\Lambda (ba)$.

If ${\rm deg}(a)=\hat{p}$ and ${\rm deg}(b)=\widehat{n-p+1}$, where
$2\leq p\leq n-1$, it is enough to check that $\Lambda (ab)=\Lambda
(ba)$ for $a=r_1^*\ot_R\ldots\ot_Rr_p^*$ and
$b=r_{p+1}^*\ot_R\ldots\ot_Rr_{n+1}^*$, where
$r_1^*,\ldots,r_{n+1}^*\in R^*$.

We see that

\begin{equation} \label{lambdaciclic}
\Lambda(r_1^*\cdot r_2^*\cdot\ldots\cdot
r_{n+1}^*)=\Lambda(r_{n+1}^*\cdot r_1^*\ldots\cdot r_n^*)
\end{equation}
Indeed, \bea \Lambda(r_1^*\cdot r_2^*\cdot\ldots\cdot
r_{n+1}^*)&=&\Lambda(\varphi(r_1^*\ot_R\ldots \ot_Rr_n^*)r_{n+1}^*)\\
&=& \Lambda(r_{n+1}^*\varphi(r_1^*\ot_R\ldots \ot_Rr_n^*))\\
&=&\Lambda(r_{n+1}^*\cdot r_1^*\cdot \ldots \cdot r_n^*).\eea

Then using $n+1-p$ times equation (\ref{lambdaciclic}), we get that
$$\Lambda(r_1^*\cdot r_2^*\cdot\ldots\cdot
r_{n+1}^*)=\Lambda(r_{p+1}^*\cdot \ldots\cdot r_{n+1}^*\cdot
r_1^*\cdot \ldots r_p^*),$$ showing that $\Lambda (ab)=\Lambda(ba)$.

(ii) $A$ is $\hat{i}$-faithful for any $0\leq i\leq n-1$, and since
$R$ is Frobenius, we have seen that $A_{\hat{i}}=(R^*)^{\ot_R\,
i}\simeq (_1R_\nu)^{\ot_R\, i}\simeq {_1R_{\nu^i}}$ as
$R$-bimodules, thus $A_{\hat{i}}\simeq R\simeq R^*$ as left
$R$-modules. These show that $A$ is $\hat{i}$-graded Frobenius. \qed
\\

{\bf Proof of Theorem D.} (i)

(i) If $n=2$, the condition $(R^*)^{\ot_R\,(n-2)}\simeq R$ is
automatically satisfied, and in this case we know that
$A(R,R^*,\varphi)$ is symmetric by \cite[Proposition 6.1]{dnn2}.

Let $n>2$ and assume that there is
 an isomorphism of
$R$-bimodules $\gamma:(R^*)^{\ot_R\,(n-2)}\ra R$; in this case $R$
must be quasi-Frobenius. By Theorem B,  $\gamma$ is associative.
Tensoring with the identity map of $R^*$, we obtain an isomorphism
of $R$-bimodules
$$\theta:A(R,R^*,\varphi)_{\widehat{n-1}}=(R^*)^{\ot_R\, (n-1)}\ra
R^*,\; \theta(r_1^*\ot_R\ldots\ot_Rr_{n-1}^*)=r_1^*\lhu
\gamma(r_2^*\ot_R\ldots\ot_Rr_{n-1}^*).$$ Since $A(R,R^*,\varphi)$
is $\widehat{n-1}$-faithful by Proposition A (note that $R^*$ is
invertible since $(R^*)^{\ot_R\,(n-2)}\simeq R$), we get that
$A(R,R^*,\varphi)$ is $\widehat{n-1}$-graded Frobenius,  a Frobenius
form $\Lambda:A(R,R^*,\varphi)\ra K$ is given by \bea
\Lambda(r_1^*\ot_R\ldots\ot_Rr_{n-1}^*)&=&(r_1^*\lhu
\gamma(r_2^*\ot_R\ldots\ot_Rr_{n-1}^*)(1)\\
&=&r_1^*(\gamma(r_2^*\ot_R\ldots\ot_Rr_{n-1}^*)),\eea and $\Lambda$
vanishes on any homogeneous component of $A(R,R^*,\varphi)$ of
degree $\neq \widehat{n-1}$. Now we see that \bea
\Lambda(r_1^*\ot_R\ldots\ot_Rr_{n-1}^*)&=&(r_1^*\lhu
\gamma(r_2^*\ot_R\ldots\ot_Rr_{n-1}^*)(1)\\
&=&(\gamma(r_1^*\ot_R\ldots\ot_Rr_{n-2}^*)\lhu r_{n-1}^*)(1) \\
&=&r_{n-1}^*(\gamma(r_1^*\ot_R\ldots\ot_Rr_{n-2}^*))\\
&=&\Lambda(r_{n-1}^*\ot_Rr_1^*\ot_R\ldots\ot_Rr_{n-2}^*).\eea Note
that the second equality holds since $\gamma$ is associative.
Applying the formula obtained above $i$ times, we get
$\Lambda(r_{n-i}^*\ot_R\ldots r_{n-1}^*\ot_Rr_1^*\ot_R\ldots
\ot_Rr_{n-i-1}^*)=\Lambda(r_1^*\ot_R\ldots\ot_Rr_{n-1}^*)$ for any
$1\leq i<n-1$, which implies that $\Lambda(ab)=\Lambda(ba)$ for any
homogeneous $a,b\in A(R,R^*,\varphi)$ such that $ab\in
A(R,R^*,\varphi)_{\widehat{n-1}}$. We conclude that
$A(R,R^*,\varphi)$ is symmetric.

(ii) Since $R$ is Frobenius, there is an isomorphism of bimodules
$$\theta_0:(R^*)^{\ot_R\, (n-1)}\ra {_1R_{\nu^{n-1}}}, \theta ((r_1\rhu
\lambda)\ot_R\ldots\ot_R(r_{n-1}\rhu
\lambda))=r_1\nu(r_2)\cdots\nu^{n-2}(r_{n-1}).$$ We regard this as
an isomorphism of left $R$-modules. Composing with the isomorphism
of left $R$-modules $R\simeq R^*$, $r\mapsto (r\rhu \lambda)$, we
obtain an isomorphism of left $R$-modules
$$\theta:(R^*)^{\ot_R\, (n-1)}\ra R, \theta ((r_1\rhu
\lambda)\ot_R\ldots\ot_R(r_{n-1}\rhu
\lambda))=(r_1\nu(r_2)\cdots\nu^{n-2}(r_{n-1}))\rhu \lambda.$$ Since
$A(R,R^*,\varphi)$ is $\widehat{n-1}$-faithful, we see that
$A(R,R^*,\varphi)$ is $\widehat{n-1}$-graded Frobenius, and a
Frobenius form $\Lambda:A(R,R^*,\varphi)\ra K$ is given by \bea
\Lambda(r_1^*\ot_R\ldots\ot_Rr_{n-1}^*)&=&\theta(r_1^*\ot_R\ldots\ot_Rr_{n-1}^*)(1)\\
&=&\lambda(r_1\nu(r_2)\cdots\nu^{n-2}(r_{n-1})),\eea and $\Lambda$
vanishes on homogeneous components of degree $\neq \widehat{n-1}$.

Now for any $1\leq i<n-1$, and any $r_1,\ldots,r_i,s_1,\ldots,
s_{n-1-i}\in R$, we have

\begin{flalign*}
&\Lambda((r_1\rhu \lambda)\ot_R\ldots\ot_R(r_i\rhu
\lambda)\ot_R(s_1\rhu \lambda)\ot_R\ldots\ot_R(s_{n-1-i}\rhu
\lambda))=&\\
&=\lambda(r_1\nu(r_2)\cdots
\nu^{i-1}(r_i)\nu^i(s_1)\cdots\nu^{n-2}(s_{n-1-i}))&\\
&=\lambda(\nu^i(s_1)\cdots\nu^{n-2}(s_{n-1-i})
\nu(r_1)\nu^2(r_2)\cdots \nu^{i}(r_i))  &\\
&= \lambda( s_1\cdots
\nu^{n-2-i}(s_{n-1-i})\nu^{1-i}(r_1)\nu^{2-i}(r_2)\cdots
\nu(r_{i-1})r_i)\;\;\;\;\;\; \mbox{(since }\lambda\nu^i=\lambda \mbox{)}&\\
&=\Lambda((s_1\rhu \lambda)\ot_R\ldots\ot_R(s_{n-1-i}\rhu
\lambda)\ot_R (\nu^{2-n}(r_1)\rhu \lambda)\ot_R\ldots\ot_R
(\nu^{2-n}(r_i)\rhu \lambda),&
\end{flalign*}
showing that the Nakayama automorphism $\mathcal{N}$ associated with
$\Lambda$ works as
$$\mathcal{N}((r_1\rhu \lambda)\ot_R\ldots\ot_R(r_{i}\rhu
\lambda))=(\nu^{2-n}(r_1)\rhu
\lambda)\ot_R\ldots\ot_R(\nu^{2-n}(r_{i})\rhu \lambda).$$

On the other hand, if $r\in R$, then for any $r_1,\ldots,r_{n-1}\in
R$

\bea \Lambda(r\cdot (r_1\rhu \lambda)\ot_R\ldots\ot_R(r_{n-1}\rhu
\lambda))&=& \Lambda((rr_1\rhu \lambda)\ot_R\ldots\ot_R(r_{n-1}\rhu
\lambda))\\
&=&\lambda(rr_1\nu(r_2)\cdots \nu^{n-2}(r_{n-1}))\\
&=& \lambda(r_1\nu(r_2)\cdots \nu^{n-2}(r_{n-1})\nu(r))\\
&=&  \lambda(r_1\nu(r_2)\cdots \nu^{n-2}(r_{n-1}\nu^{3-n}(r)))\\
&=& \Lambda((r_1\rhu
\lambda)\ot_R\ldots\ot_R(r_{n-1}\nu^{3-n}(r)\rhu \lambda))\\
&=& \Lambda((r_1\rhu \lambda)\ot_R\ldots\ot_R((r_{n-1}\rhu
\lambda)\lhu \nu^{2-n}(r)))\\
&=&\Lambda((r_1\rhu \lambda)\ot_R\ldots\ot_R(r_{n-1}\rhu
\lambda)\cdot \nu^{2-n}(r)),\eea showing that
$\mathcal{N}(r)=\nu^{2-n}(r)$.
 \qed

\begin{remark} \label{remarcaOre}
{\rm

Let $R$ be a Frobenius algebra with a Frobenius form $\lambda$ and
associated Nakayama automorphism $\nu$, $n\geq 2$, and
$$\varphi:(R^*)^{\ot_R\, n}\ra R, \varphi((r_1\rhu
\lambda)\ot_R\ldots\ot_R(r_n\rhu
\lambda))=r_1\nu(r_2)\cdots\nu^{n-1}(r_n)c$$ an associative bimodule
morphism, where $c\in R$ satisfies $\nu(c)=c$ and $\nu^n(r)c=cr$ for
any $r\in R$.

For any $1\leq j\leq n-1$ using the isomorphism of bimodules
$(R^*)^{\ot_R\, j}\simeq {_1R_{\nu^j}}$, any element of
$(R^*)^{\ot_R\, j}$ can be uniquely written as $(r\rhu
\lambda)\ot_R\lambda\ot_R\ldots\ot_R\lambda$ for some $r\in R$;
denote by $r^{[j]}$ this element of $(R^*)^{\ot_R\, j}$. We also
denote any $r\in R$ by $r^{[0]}$. We have
$$A(R,R^*,\varphi)=R^{[0]}\oplus R^{[1]}\oplus \ldots \oplus
R^{[n-1]},$$ and the multiplication of $A(R,R^*,\varphi)$, let it be
denoted by $*$ for avoiding confusion, can be rewritten as
$$r^{[j]}*s^{[i]}=\left\{
\begin{array}{ll}
(r\nu^j(s))^{[j+i]},&\mbox{ if } j+i<n,\\
(r\nu^j(s))^{[j+i-n]},&\mbox{ if } j+i\geq n
\end{array}
\right.$$ This shows that $A(R,R^*,\varphi)\simeq
\frac{R[X,\nu]}{(X^n-c)}$, where $R[X,\nu]$ is the Ore extension of
the algebra $R$, with respect to the automorphism $\nu$ of $R$, and
with zero derivation. An isomorphism takes $r^{[j]}$ to the class of
$rX^j$ in the factor algebra $\frac{R[X,\nu]}{(X^n-c)}$. }
\end{remark}

Now we give a criterion for an algebra of type $A(R,R^*,\varphi)$ to
be a symmetric algebra, in the case where $R$ is Frobenius.

\begin{proposition} \label{criteriu}
Let $R$ be a Frobenius algebra with a Frobenius form $\lambda$ and
associated Nakayama automorphism $\nu$, $n\geq 2$ an integer, and
$c\in R$ such that $\nu(c)=c$ and $\nu^n(r)c=cr$ for any $r\in R$.
Let $\varphi:(R^*)^{\ot_R\, n}\ra R$, $\varphi((r_1\rhu
\lambda)\ot_R\ldots\ot_R(r_n\rhu
\lambda))=r_1\nu(r_2)\cdots\nu^{n-1}(r_n)c$. Then $A(R,R^*,\varphi)$
is a symmetric algebra if and only if there exist
$r_0,\ldots,r_{n-1},s_0,\ldots,s_{n-1}\in R$ such that the following
three conditions hold.

$\rm{(I)}$ $r_js=\nu^{n-j-2}(s)r_j$ for any $0\leq j\leq n-1$ and
any $s\in R$;

$\rm{(II)}$ $\nu(r_j)=r_j$ for any $0\leq j<n-1$, and
$\nu(r_{n-1})c=r_{n-1}c$;

$\rm{(III)}$ The following equations are satisfied

\small{ $\begin{array}{lllllllllll}
r_0s_0&+&r_1\nu(s_{n-1})c&+&r_2\nu^2(s_{n-2})c&+&\ldots
&+&r_{n-1}\nu^{n-1}(s_1)c&=1\\
r_0s_1&+&r_1\nu(s_0)&+&r_2\nu^2(s_{n-1})c&+&\ldots
&+&r_{n-1}\nu^{n-1}(s_2)c&=0\\
r_0s_2&+&r_1\nu(s_1)&+&r_2\nu^2(s_0)&+&r_3\nu^3(s_{n-1})c+\ldots&
+&r_{n-1}\nu^{n-1}(s_3)c&=0\\
\ldots&\ldots&\ldots&\ldots&\ldots&\ldots&\ldots&\ldots&\ldots&\ldots\\
r_0s_{n-2}&+&r_1\nu(s_{n-3})&+&\ldots &\ldots&\ldots
+r_{n-2}\nu^{n-2}(s_0)&+&r_{n-1}\nu^{n-1}(s_{n-1})c&=0\\
r_0s_{n-1}&+&r_1\nu(s_{n-2})&+& \ldots&\ldots&\ldots
&+&r_{n-1}\nu^{n-1}(s_0)&=0
 \end{array}$}
\end{proposition}
\begin{proof}
Denote $A(R,R^*,\varphi)$ by $A$. It is
 a Frobenius algebra, with a Frobenius form $\Lambda$, and associated Nakayama
automorphism $\mathcal{N}$. Then $A$ is symmetric if and only if
$\mathcal{N}$ is an inner automorphism, i.e., there exists an
invertible $U\in A$ such that $\mathcal{N}(a)=U^{-1}aU$, or
equivalently, $U\mathcal{N}(a)=aU$, for any $a\in A$. Since
$\mathcal{N}$ preserves the degree, this is further equivalent to
the fact that each homogeneous component $U_{\hat{j}}\in
A_{\hat{j}}$ of $U$ satisfies
$U_{\hat{j}}\mathcal{N}(a)=aU_{\hat{j}}$, for any $a\in A$. As in
Remark \ref{remarcaOre}, for any $1\leq j\leq n-1$  we can write
$U_{\hat{j}}=(r_j\rhu \lambda)\ot_R\lambda\ot_R\ldots\ot_R\lambda$
for some $r_j\in R$. We also denote $U_{\hat{0}}=r_0$. We seek $r_j$
such that $U_{\hat{j}}\mathcal{N}(a)=aU_{\hat{j}}$ for any
homogeneous $a\in A$. For $a$ of degree $\hat{i}$, we can similarly
write $a=(s\rhu \lambda)\ot_R\lambda\ot_R\ldots\ot_R\lambda$ ($i$
tensor factors) for some $s\in R$, and then
$U_{\hat{j}}\mathcal{N}(a)=aU_{\hat{j}}$ is equivalent to
\begin{equation} \label{comutare1}
r_j\nu^{j-n+2}(s)=s\nu^i(r_j) \mbox{  if }i+j<n,
\end{equation}
and to
\begin{equation} \label{comutare2}
r_j\nu^{j-n+2}(s)c=s\nu^i(r_j)c \mbox{  if } i+j\geq n.
\end{equation}

For $j=0$, one can similarly see that
$U_{\hat{0}}\mathcal{N}(a)=aU_{\hat{0}}$ is equivalent to equation
(\ref{comutare1}) written for $j=0$.

For $0\leq j<n-1$, if we take $s=1$ and $i=1$ in equation
(\ref{comutare1}), we get $\nu(r_j)=r_j$, while for $j=n-1$, if we
take $s=1$ and $i=1$ in equation (\ref{comutare2}), we get
$\nu(r_{n-1})c=r_{n-1}c$. We have obtained that $U\mathcal{N}(a)=aU$
for any $a\in A$ if and only if $r_0,\ldots,r_{n-1}$ satisfy
conditions (I) and (II).

Finally, we show that the invertibility of $U$ is equivalent to
condition (III). Indeed, $U$ is invertible if and only if there
exists
$$V=s_0+(s_1\rhu\lambda)+(s_2\rhu\lambda)\ot_R\lambda +\ldots
+(s_{n-1}\rhu \lambda)\ot_R\lambda\ot_R\ldots\ot_R\lambda \in A$$
such that $UV=1$. Equating the homogeneous components of degree
$\hat{0},\hat{1},\ldots,\widehat{n-1}$ with $1,0,\ldots,0$, we
obtain exactly the equations in (III), and this ends the proof.
\end{proof}

\begin{remark}
{\rm  We can explicitly give the solution of the system in condition
(III) of the previous proposition in two cases where we already know
that $A(R,R^*,\varphi)$ is symmetric.

(i) If $\varphi$ is an isomorphism, then $\varphi$ is given by some
invertible element $c$ in $R$ satisfying $\nu^n(r)c=cr$ for any
$r\in R$, and $\nu(c)=c$. Take $r_{n-2}=1$, which satisfies the
conditions $r_{n-2}s=sr_{n-2}$ for any $s\in R$, and
$\nu(r_{n-2})=r_{n-2}$, required in (I) and (II), $s_2=c^{-1}$, and
$r_i=0$ for any $i\neq n-2$, and $s_i=0$ for any $i\neq 2$. Then
$r_0,\ldots,r_{n-1},s_0,\ldots,s_{n-1}$ satisfy (I), (II) and (III).

(ii) If $(R^*)^{\ot_R\, (n-2)}\simeq R$ as $R$-bimodules, then
$\nu^{n-2}$ is an inner automorphism. Let $u\in R$ be an invertible
element such that $\nu^{n-2}(r)=u^{-1}ru$ for any $r\in R$. By Lemma
\ref{lemanuu}, $\nu(u)=u$. Take $r_0=u^{-1}$, $s_0=u$, and
$r_i=s_i=0$ for any $i\neq 0$. Then
$r_0,\ldots,r_{n-1},s_0,\ldots,s_{n-1}$ satisfy (I), (II) and (III).
}
\end{remark}

Now we use Theorem D and the criterion in Proposition \ref{criteriu}
for proving Corollary E.\\

{\bf Proof of Corollary E.} (i) If $(R^*)^{\ot_R\, (n-2)}\simeq R$
or $\varphi$ is an isomorphism, then $A(R,R^*,\varphi)$ is symmetric
by Theorem D (i), respectively by Theorem C (i); note that the
condition that $R$ is local is not necessary.

Conversely assume that $A(R,R^*,\varphi)$ is symmetric. By
Proposition \ref{criteriu} there exist $r_0,\ldots,r_{n-1}$,
$s_0,\ldots,s_{n-1}\in R$ satisfying conditions (I), (II) and (III),
where $c$ is the element defining $\varphi$, as in the statement of
the proposition. Since $R$ is local, the first equation in (III)
shows that either $r_0s_0$is invertible, or one of
$r_1\nu(s_{n-1})c, r_2\nu^2(s_{n-2})c,\ldots,r_{n-1}\nu^{n-1}(s_1)c$
is invertible. In the second case, we get that $c$ is invertible,
and then $\varphi$ is an isomorphism. If $r_0s_0$ is invertible,
then so is $r_0$, and condition (I) for $j=0$ gives
$\nu^{n-2}(s)=r_0sr_0^{-1}$ for any $s\in R$, which shows that
$\nu^{n-2}$ is inner, or equivalently, $(R^*)^{\ot_R\, (n-2)}\simeq
R$ as $R$-bimodules.

(ii) Again, the if implication follows from Theorem D. Conversely,
assume that $A(R,n)$ is symmetric. As $\varphi=0$, the element $c$
defining $\varphi$ is also 0, and then the first equation in (III)
is just $r_0s_0=1$, showing that $r_0$ is invertible. This implies
that $\nu^{n-2}$ is inner, thus $(R^*)^{\ot_R\, (n-2)}\simeq R$.
\qed

\begin{remark}
{\rm  The assumption that $R$ is local can not be omitted in
Corollary E (i). Indeed, let $n\geq 3$, and let $S$ and $T$ be
Frobenius algebras such that the classes of $S^*$ and $T^*$ have
orders $n$, respectively $n-2$ in the Picard groups of $S$,
respectively $T$. Let $\rho,\omega$ be Frobenius forms in $S,T$, and
$\mu,\delta$ the associated Nakayama automorphisms.

Consider the algebra $R=S\times T$, which is Frobenius. A Frobenius
form on $R$ is $\lambda:R\ra K$, $\lambda (s,t)=\rho(s)+\omega(t)$
for any $s\in S,t\in T$, and the Nakayama automorphism of $R$
associated with $\lambda$ is $\nu(s,t)=(\mu(s),\delta(t))$. Clearly,
$\nu^m$ is inner for some $m$ if and only if both $\mu^m$ and
$\delta^m$ are inner. In particular, $\nu^{n-2}$ is not inner, since
neither is $\mu^{n-2}$.

By the discussion at the beginning of Section
\ref{sectionassociative}, an associative morphism of $R$-bimodules
$\varphi:(R^*)^{\ot_R\, n}\ra R$ is given by a pair $(c,d)\in
S\times T$ such that $\nu(c,d)=(c,d)$ and $(\nu^n(s,t))\cdot
(c,d)=(c,d)\cdot (s,t)$ for any $s\in S,t\in T$, which is the same
with $\mu^n(s)c=cs$ for any $s\in S$, $\mu(c)=c$, $\delta^n(t)d=dt$
for any $t\in T$, and $\delta(d)=d$. Choose such a pair $(c,d)$ such
that $c$ is an invertible element of $S$ defining an isomorphism of
$S$-bimodules $\psi:(S^*)^{\ot_S\, n}\ra S$ (which exists by the
assumption on $S$), and $d=0$, which obviously defines the zero
morphism of $T$-bimodules $\xi:(T^*)^{\ot_T\, n}\ra T$.

We apply the criterion in Proposition \ref{criteriu} in order to
check whether $A(R,R^*,\varphi)$ is symmetric. The system of
equations from (III) in the algebra $R$, written for the element
$(c,d)\in R$, splits into the similar system of equations in $S$ for
the element $c$, and the one in $R$ for the element $d=0$, while
conditions (I) and (II) also split into the two cases. Applying the
criterion, the first system in $S$ has solution, since
$A(S,S^*,\psi)$ is symmetric by Theorem C, and so does the second
one in $T$, since $A(T,T^*,\xi)=A(T,n)$ is symmetric by Theorem D
(i). Putting together the solutions of these two systems, we get
that the initial system in $R$ has solution, so $A(R,R^*,\varphi)$
is symmetric.

On the other hand, neither $\varphi$ is an isomorphism, since
$(c,d)$ is not invertible, nor $(T^*)^{\ot_T\, (n-2)}\simeq T$ as
$T$-bimodules, since $\nu^{n-2}$ is not inner. }
\end{remark}

Before starting the proof of Theorem F, let $R$ be quasi-Frobenius,
and let $e_1,\ldots,e_q$ be primitive idempotents of $R$ such that
$Re_1,\ldots,Re_q$ is a system of representatives for the
isomorphism types of principal indecomposable $R$-modules, with
multiplicities $m_1,\ldots,m_q$ in $R$; thus $R\simeq
(Re_1)^{m_1}\oplus \ldots \oplus (Re_q)^{m_q}$ as left $R$-modules.
Denote by ${\rm top}(Re_1)=S_1,\ldots,{\rm top}(Re_q)=S_q$ the
(distinct) isomorphism types of simple $R$-modules. The associated
Nakayama permutation $\pi$ of the set $\{ 1,\ldots,q\}$ is defined
such that ${\rm top}(Re_i)={\rm soc}(Re_{\pi(i)})$ for any $1\leq
i\leq q$, see \cite[Theorem IV.6.1]{sy}. One knows that $R$ is
Frobenius if and only if $m_i=m_{\pi(i)}$ for any $1\leq i\leq q$,
see \cite[Theorem IV.6.2]{sy}.

\begin{lemma} \label{lemaR*tensorsimplu}
Keeping the notation above, let $e$ be a primitive idempotent of
$R$. Then ${\rm top}(Re)\simeq R^*\ot_R{\rm soc}(Re)$. In
particular, $R^*\ot S_j\simeq S_{\pi(j)}$ for any $1\leq j\leq q$.
\end{lemma}
\begin{proof}
We have ${\rm top}(Re)\simeq {\rm soc}((eR)^*)$ by \cite[Lemma
I.8.22]{sy}. Since $(eR)^*\simeq R^*e\simeq R^*\ot_RRe$ as left
$R$-modules, we see that ${\rm soc}((eR)^*)\simeq R^*\ot_R{\rm
soc}(Re)$, showing that ${\rm top}(Re)\simeq R^*\ot_R{\rm soc}(Re)$.

For $e=e_i$, this means that $S_i\simeq R^*\ot_R S_{\pi^{-1}(i)}$.
Thus $R^*\ot S_j\simeq S_{\pi(j)}$ for any $j$.
\end{proof}

{\bf Proof of Theorem F.} Denote $A=A(R,n)$ and
$I=A_{\hat{1}}+\ldots +A_{\widehat{n-1}}$, which is a nilpotent
ideal of $A$. Then $J(A)=J(R)+I$. Since $A/I\simeq R$, any
$R$-module $M$ has also an $A$-module structure such that $IM=0$ and
the action of $A_{\hat{0}}$ on $M$ is given by the initial
$R$-module structure. Moreover, two $R$-modules $M$ and $N$ are
isomorphic if and only if $M\simeq N$ as $A$-modules. We will freely
regard an $R$-module as either an $R$-module or as an $A$-module;
note that the submodules are the same in the two situations.

We see that $A_{\widehat{n-1}}$ is a graded $A$-submodule of $A$,
and it is graded essential in $A$. Indeed, by Proposition A (i) for
$M=R^*$, $A$ is $\widehat{n-1}$-faithful, so for any $0\neq z\in
A_{\hat{i}}$, where $1\leq i\leq n-1$, we have
$A_{\widehat{n-1-i}}z\neq 0$, thus there is $a\in
A_{\widehat{n-1-i}}$ such that $0\neq az\in A_{\widehat{n-1}}$. By
\cite[Proposition 2.3.5]{nvo}, $A_{\widehat{n-1}}$ is an essential
$A$-submodule of $A$. We note that this implies that the graded
socle of $A$ (i.e. the socle of the graded left $A$-module $A$)
coincides with the socle of $A$, and it lies inside
$A_{\widehat{n-1}}$, more precisely, it is just the socle of the
$R$-module $A_{\widehat{n-1}}$.

If $e$ is a primitive idempotent of $R$, then it is also primitive
in $A$. Indeed, if $e=a+b$ for some idempotents $a,b$, then
$a_{\hat{0}}$ and $b_{\hat{0}}$ are idempotents of $R$ and
$e=a_{\hat{0}}+b_{\hat{0}}$, so one of $a_{\hat{0}}$ and
$b_{\hat{0}}$ must be $0$. If $a_{\hat{0}}=0$, then equating the
homogeneous components of degrees $\hat{1},\ldots,\widehat{n-1}$ in
$a^2=a$, we successively get
$a_{\hat{1}}=0,\ldots,a_{\widehat{n-1}}=0$, thus $a=0$. Therefore,
the isomorphism types of principal indecomposable $A$-modules are
$Ae_1,\ldots,Ae_q$, and their multiplicities in $A$ are
$m_1,\ldots,m_q$.

As above, denote ${\rm top}(Re_1)=S_1,\ldots,{\rm top}(Re_q)=S_q$.

If $e$ is a primitive idempotent of $R$, consider ${\rm top}(Ae)$
and ${\rm soc}(Ae)$ of the left $A$-module $Ae$. Then ${\rm
top}(Ae)=\frac{Ae}{J(A)e}$ is annihilated by $I$, so it can be
regarded as a left $R$-module, and ${\rm top}(Ae)\simeq
\frac{Re}{J(R)e}={\rm top}(_ARe)={\rm top}(_RRe)$ as left
$R$-modules. On the other hand,
$(Ae)_{\widehat{n-1}}=A_{\widehat{n-1}}\cap Ae$ is an essential
submodule of $Ae$, so ${\rm soc}(Ae)={\rm
soc}((Ae)_{\widehat{n-1}})$. Clearly, $(Ae)_{\widehat{n-1}}$ is
annihilated by $I$, and then so is ${\rm soc}(Ae)$, thus we can
regard ${\rm soc}(Ae)$ as a left module over $A/I\simeq R$.
Moreover, the socle of $(Ae)_{\widehat{n-1}}$
 as an $A$-module
coincides with its socle as an $R$-module. We have
$(Ae)_{\widehat{n-1}}=(R^*)^{\ot_R\,(n-1)}e\simeq
(R^*)^{\ot_R\,(n-1)}\ot_RRe$ as left $R$-modules, and since
$(R^*)^{\ot_R\,(n-1)}$ is an invertible $R$-bimodule, there is an
isomorphism of left $R$-modules ${\rm
soc}((R^*)^{\ot_R\,(n-1)}\ot_RRe)\simeq (R^*)^{\ot_R\,(n-1)}\ot_R\,
{\rm soc}(Re)$. We conclude that ${\rm soc}(Ae)\simeq
(R^*)^{\ot_R\,(n-1)}\ot_R\, {\rm soc}(Re)$ as left $R$-modules.

For $e=e_i$, these show that there are $R$-module isomorphisms ${\rm
top}(Ae_i)\simeq {\rm top}(Re_i)=S_i$ and $${\rm soc}(Ae_i)\simeq
(R^*)^{\ot_R\,(n-1)}\ot_R\, {\rm soc}(Re_i)\simeq
(R^*)^{\ot_R\,(n-1)}\ot_R\,S_{\pi^{-1}(i)}.$$ Using Lemma
\ref{lemaR*tensorsimplu}, the latter one gives ${\rm
soc}(Ae_i)\simeq S_{\pi^{n-2}(i)}$. As a consequence, ${\rm
top}(Ae_i)\simeq {\rm soc}(Ae_{\pi^{2-n}(i)})$ as $R$-modules, and
as $A$-modules as well, showing that the Nakayama permutation of the
quasi-Frobenius algebra $A$ is $\pi^{2-n}$. By \cite[Theorem
IV.6.2]{sy}, we conclude that $A$ is Frobenius if and only if
$m_i=m_{\pi^{n-2}(i)}$ for any $1\leq i\leq q$. \qed

\section{Some examples}

We have seen that $A(R,2)$ is symmetric for any $R$. In this section
we present several examples to illustrate the behaviour of algebras
of type $A(R,n)$ for $n\geq 3$, with respect to Frobenius type
properties. By Proposition A (ii), if $R$ is quasi-Frobenius, then
so is $A(R,n)$ for any $n$. We give below examples where $A(R,n)$ is
not Frobenius, and some others where $A(R,n)$ is Frobenius.

\begin{example} \label{ex1}
{\rm  Let $p,q$ be positive integers with $p\neq q$. Let $R$ be the
example of a quasi-Frobenius algebra which is not Frobenius,
described in \cite[Section 6]{dnn3}. $R$ has a basis
$$(E_{ij})_{1\leq i,j\leq p}, (X_{ir})_{1\leq i\leq p \atop 1\leq
r\leq q}, (Y_{ri})_{1\leq r\leq q \atop 1\leq i\leq p},
(F_{rt})_{1\leq r,t\leq q},$$ with relations \bea
E_{ij}E_{js}=E_{is},&
F_{rt}F_{tu}=F_{ru}\\
E_{ij}X_{jr}=X_{ir}, &X_{ir}F_{rt}=X_{it}\\
F_{rt}Y_{ti}=Y_{ri},& Y_{ri}E_{ij}=Y_{rj} \eea for any $1\leq
i,j,s\leq p$, $1\leq r,t,u\leq q$, and any other product of two
elements in the basis is zero. In the case where $p=2$ and $q=1$,
$R$ is Nakayama's example of a 9-dimensional algebra which is
quasi-Frobenius, but not Frobenius, see \cite[Example 16.19
(5)]{lam}.

For any $1\leq j\leq p$, let $U_j=\langle E_{ij},Y_{rj}|1\leq i\leq
p, 1\leq r\leq q\rangle$, and for any $1\leq t\leq q$, let
$U'_t=\langle X_{it},F_{rt}|1\leq i\leq p, 1\leq r\leq q\rangle$.
Then $R=(\oplus_{1\leq j\leq p}U_j) \oplus (\oplus_{1\leq t\leq
n}U'_t)$ is a decomposition into a direct sum of indecomposable left
$R$-modules, $U_1\simeq \ldots\simeq U_p$ and $U'_1\simeq \ldots
\simeq U'_q$, while $U_1$ and $U'_1$ are not isomorphic. Thus the
multiplicities of the two types of principal indecomposables are
$m_1=p$ and $m_2=q$. On the other hand, ${\rm top}(U_j)\simeq {\rm
soc}(U'_t)$ and ${\rm top}(U'_t)\simeq {\rm soc}(U_j)$, so the
Nakayama permutation $\pi$ is the transposition $(1\; 2)$.

By Proposition A (ii), $A(R,n)$ is quasi-Frobenius. As $m_1\neq
m_2$, it follows from Theorem F that $A(R,n)$ is Frobenius if $n$ is
even, while $A(R,n)$ is not Frobenius in the case where $n$ is odd.
In fact, for even $n$, $A(R,n)$ is even symmetric. Indeed, a basic
algebra of $R$ is $S=(E_{11}+F_{11})R(E_{11}+F_{11})$, and it is
showed in \cite[Proposition 4.4]{dnn2} that the class of $S^*$ in
the Picard groups of $S$ has order 2, and then by \cite[Corollary
3.4]{dnn2}, the class $R^*$ has also order 2 in ${\rm Pic}(R)$, thus
$R^*\ot_RR^*\simeq R$. Then $(R^*)^{\ot_R\, (n-2)}\simeq R$ for any
even $n$, thus $A(R,n)$ is symmetric. }
\end{example}

Before going to the next example, we note that by the tensor-Hom
adjunction,  for any finite-dimensional $R$-modules $M$ and $N$,
there is a natural isomorphism of $R$-bimodules $(N\ot_RM)^*\simeq
{\rm Hom}_{R-}(M,N^*)$. In particular, if $R$ is a
finite-dimensional algebra, taking $M=N=R^*$, we get an isomorphism
of $R$-bimodules

\begin{equation}\label{izoR*R*}
(R^*\ot_RR^*)^*\simeq {\rm Hom}_{R-}(R^*,R).
\end{equation}

The two examples below show that if $R$ is not quasi-Frobenius, then
$A(R,n)$ may not be quasi-Frobenius, but it is also possible that it
is even symmetric.

\begin{example}
{\rm Let $R=\small{\left(
\begin{array}{ll}K&K\\0&K\end{array}\right)}$ be the algebra of
upper triangular $2\times 2$ matrices. Denoting by $e,f$ the
diagonal matrix units and by $x$ the off-diagonal one, $R=\langle
e,f,x\rangle$, with relations $e^2=e,f^2=f,ex=xf=x$, and any other
product of two elements of the basis $\{e,f,x\}$ is zero. If $\{
e^*,f^*,x^*\}$ is the dual basis of $R^*$, the left and right
actions of $R$ on $R^*$ are given by
$$e\rhu e^*=e^*\lhu e=e^*, f\rhu f^*=f^*\lhu f=f^*,$$
$$f\rhu x^*=x^*\lhu e=x^*, x\rhu x^*=e^*, x^*\lhu x=f^*,$$
and any other action of an element in $\{e,f,x\}$ on an element in
the dual basis is zero.

The decomposition of the left $R$-module $R$ into a sum
indecomposables is $R =S\oplus P$, where $S=\langle e\rangle$ is a
simple module, and $P=\langle f,x\rangle$. The socle of $P$ is
$\langle x\rangle\simeq S$. An indecomposable decomposition of the
left $R$-module $R^*$ is $R^*=S'\oplus P'$, where $S'=\langle
f^*\rangle$ is a simple module not isomorphic to $S$, and
$P'=\langle e^*,x^*\rangle$, which is isomorphic to $P$ (an
isomorphism $P\ra P'$ takes $x$ to $e^*$, and $f$ to $x^*$). Note
that $S'$ is injective, as a direct summand of $R^*$. The Jacobson
radical of $R$ is $J(R)=\langle x\rangle$, and $P/J(R)P\simeq S'$,
thus $P\simeq P'$ is a projective cover of $S'$. Then by
\cite[Proposition 2.8]{lorenz}, $${\rm dim}_K\,{\rm
Hom}_R(P',R)=\mu(S',R){\rm dim}_K\,({\rm End}_R(S')),$$ where
$\mu(S',R)=1$ is the multiplicity of $S'$ as a composition factor of
$_RR$, and ${\rm End}_R(S')$ is the endomorphism algebra of $S'$. We
obtain that ${\rm Hom}_R(P',R)$ has dimension 1. We also see that
${\rm Hom}_R(S',R)=0$, since the socle of $_RR$ is $S\oplus \langle
x\rangle\simeq S^2$. Therefore ${\rm dim}_K\,{\rm Hom}_R(R^*,R)=1$.

Using (\ref{izoR*R*}), we see that ${\rm dim}_K(R^*\ot_RR^*)=1$. We
have $f^*\ot_Rx^*=(x^*\lhu x)\ot_Rx^*=x^*\ot_R(x\rhu
x^*)=x^*\ot_Re^*$, and it is easily checked that any other
$u^*\ot_Rv^*$ with $u^*,v^*\in \{ e^*,f^*,x^*\}$ is zero; for
example $e^*\ot_Re^*=e^*\ot_R(x\rhu x^*)=(e^*\lhu x)\ot_Rx^*=0$.
Therefore $\{ f^*\ot_Rx^*\}$ is a basis of $R^*\ot_RR^*$. As a
consequence, $R^*\ot_RR^*\ot_RR^*=0$. Indeed,
$f^*\ot_Rx^*\ot_Re^*=x^*\ot_Re^*\ot_Re^*=0$, and clearly
$f^*\ot_Rx^*\ot_Rf^*=f^*\ot_Rx^*\ot_Rx^*=0$.

Consider the $\mathbb{Z}_3$-graded algebra $A=A(R,3)$. We have
$A=R\oplus R^*\oplus (R^*\ot_RR^*)$. We will show that $A$ is
 not quasi-Frobenius.

If we denote $z=f^*\ot_Rx^*$, the left action of $R$ on
$R^*\ot_RR^*$ is given by $f\rhu z=z, e\rhu z=x\rhu z=0$, so
$R^*\ot_RR^*\simeq S'$ as left $R$-modules, while the right action
is $z\lhu e=z$, $z\lhu f=z\lhu x=0$.

We note that ${\rm Hom}_{R}(A,S')\simeq {\rm Hom}_R(S\oplus
P^2\oplus (S')^2,S')$ has dimension 4. Indeed, this follows from
${\rm Hom}_{R}(S,S')=0$, ${\rm Hom}_{R}(S',S')\simeq K$, and ${\rm
dim}_K{\rm Hom}_{R}(P,S')=1$; the latter holds true since $P$ is a
projective cover of $S'$.

Let $M=Ae=\langle e\rangle + \langle e^*,x^*\rangle + \langle
x^*\ot_Re^*\rangle$ and $N=Af=\langle f,x\rangle + \langle
f^*\rangle$, which are graded $A$-submodules of $A$ (the spaces
indicated in each sum are the homogeneous components; the
homogeneous component of degree $\hat{2}$ of $N$ is 0). We have
$A=M\oplus N$.

It is easily checked that $N$ is $\hat{1}$-faithful, so $N$ embeds
as an essential graded submodule into ${\rm
Coind}(N_{\hat{1}})(\hat{2})$. As $N_{\hat{1}}=S'$ is an injective
$R$-module,  ${\rm Coind}(N_{\hat{1}})(\hat{2})$ is injective, thus
an injective envelope of $N$ in the category of graded $A$-modules.
Now ${\rm dim}_K\,  {\rm Coind}(N_{\hat{1}})(\hat{2})={\rm dim}_K\,
{\rm Hom}_{R-}(A,S')=4$, and ${\rm dim}_K\, N=3$, showing that $N$
is not injective. We conclude that neither $A$ is injective, so $A$
is not quasi-Frobenius. }
\end{example}

\begin{example}  \label{ex3}
{\rm Let $R=\small{\left(
\begin{array}{ll}K&K^2\\0&K\end{array}\right)}$ be the generalized matrix algebra associated with the $K$-bimodule $K^2$. If $e,f$ are the
diagonal matrix units and $x,y$ are the matrices corresponding to
the elements of the canonical basis of the off-diagonal position,
then $R=\langle e,f,x,y\rangle$, with relations
$e^2=e,f^2=f,ex=xf=x, ey=yf=y$, and any other product of two
elements of the basis $\{e,f,x,y\}$ is zero.

The decomposition of the left $R$-module $R$ into a sum
indecomposables is $R =Re\oplus Rf$, and $Re=\langle e\rangle$,
$Rf=\langle f,x,y\rangle$. The socle of the principal indecomposable
$Rf$ is not simple; indeed, this socle is $\langle x,y\rangle
=\langle x\rangle \oplus \langle y\rangle$, a sum of two simples. By
\cite[Theorem 16.4]{lam}, $R$ is not quasi-Frobenius.

 If $\{ e^*,f^*,x^*,y^*\}$
is the dual basis of $R^*$, the left action of $R$ on $R^*$ is given
by
$$e\rhu e^*=e^*, f\rhu f^*=f^*,$$
$$f\rhu x^*=x^*, x\rhu x^*=e^*, f\rhu y^*=y^*, y\rhu y^*=e^*,$$
and any other action of an element in $\{e,f,x,y\}$ on an element in
the dual basis is zero.

We show that ${\rm Hom}_{R-}(R^*,R)=0$. Indeed, if $\omega\in {\rm
Hom}_{R-}(R^*,R)$, then $\omega(x^*)=\omega(f\rhu
x^*)=f\omega(x^*)\in fR=\langle f\rangle$, and similarly
$\omega(y^*)\in \langle f\rangle$.
 Hence $\omega(e^*)=\omega(x\rhu x^*)=x\omega(x^*)\in x\langle f\rangle
\subset \langle x\rangle$, and similarly $\omega(e^*)\in  \langle
y\rangle$, so then $\omega(e^*)\in \langle x\rangle \cap \langle
y\rangle=0$, thus $\omega(e^*)=0$.

Now $x\omega(x^*)=\omega(x\rhu x^*)=\omega(e^*)=0$, showing that
$\omega(x^*)\in \lhu e,x,y\rhu$, and then
$\omega(x^*)=e\omega(x^*)$. As $e\omega(x^*)=\omega(e\rhu x^*)=0$,
we get $\omega(x^*)=0$. Similarly, $\omega(y^*)=0$.

Finally, $\omega(f^*)=\omega(f\rhu f^*)=f\omega(f^*)$, so
$\omega(f^*)\in \langle f\rangle$. Then
$\omega(f^*)=x\omega(f^*)=\omega(x\rhu f^*)=0$. We conclude that
$\omega =0$.

Now we use (\ref{izoR*R*}) and obtain that $R^*\ot_RR^*=0$. Then
$A(3,R)=R\oplus R^*$ is precisely the trivial extension $R\cosmash
R^*$, which is a symmetric algebra.
 }
\end{example}

Using Example \ref{ex3}, we can produce for any $n\geq 3$ an example
of an algebra $T$ which is not quasi-Frobenius, such that $A(T,n)$
is symmetric and $(T^*)^{\ot_T\,(n-1)}\neq 0$.

\begin{example}
{\rm Let $n\geq 3$ and let $R$ be the algebra from Example
\ref{ex3}. Let $S$ be a quasi-Frobenius algebra such that $A(S,n)$
is symmetric, for example one such that $(S^*)^{\ot_S\, (n-2)}\simeq
S$. Then $T=R\times S$ is not quasi-Frobenius and one can easily
check that $(T^*)^{\ot_T\, j}\simeq (R^*)^{\ot_R\, j}\times
(S^*)^{\ot_S\, j}$ as $T$-bimodules for any $j$, where the actions
of $R$ on $(S^*)^{\ot_S\, j}$, as well as the ones of $S$ on
$(R^*)^{\ot_R\, j}$ are trivial. It follows that there is an algebra
isomorphism $A(T,n)\simeq A(R,n)\times A(S,n)$. As both $A(R,n)$ and
$A(S,n)$ are symmetric, so is $A(T,n)$. Since
$(S^*)^{\ot_S\,(n-1)}\neq 0$, we see that $(T^*)^{\ot_T\,(n-1)}\neq
0$, too. }
\end{example}

The following produces a class of symmetric algebras from any
finite-dimensional Hopf algebra.

\begin{example}
{\rm Let $H$ be a finite-dimensional Hopf algebra with antipode $S$.
Let $[S^2]$ be the class of $S^2$ in ${\rm Out}(H)={\rm Aut}(H)/{\rm
Inn}(H)$, the group of outer automorphisms of the algebra $H$. The
order $o([S^2])$ of $[S^2]$ in this group is the least positive
integer $\ell$ such that $S^{2\ell}$ is inner. Let also
$\mathcal{G}$ be the distinguished grouplike element in $H^*$, and
denote by $o(\mathcal{G})$ its order in the group of grouplike
elements of $H^*$. It was proved in \cite[Theorem 2.9]{dnn2} that
the order of the class of the $H$-bimodule $H^*$ in the Picard group
of $H$ is the least common multiple of $o([S^2])$ and
$o(\mathcal{G})$.

For any integer $n\geq 2$, $A(H,n)$ is a Frobenius algebra by
Theorem D (ii). By Corollary E, it is a symmetric algebra if and
only if both $o([S^2])$ and $o(\mathcal{G})$ divide $n-2$. }
\end{example}

{\bf Acknowledgement.} The first author was supported by the PNRR
grant CF 44/14.11.2022 {\it Cohomological Hall algebras of smooth
surfaces and applications}.

\end{document}